\documentclass[11pt,reqno]{amsart}
\usepackage{amsfonts,amsmath,amssymb,amsthm}
\usepackage{latexsym}
\usepackage{eucal}
\usepackage[ansinew]{inputenc}
\usepackage[american]{babel}
\usepackage{amsfonts}
\usepackage{amssymb}
\usepackage{graphicx}
\newcommand{\be}{\begin{equation}}
	\newcommand{\ee}{\end{equation}}

\newtheorem{teo}{Theorem}[section]
\newtheorem{lema}{Lemma}[section]
\newtheorem{prop}{Proposition}[section]

\newtheorem{obs}{Remark}[section]
\newtheorem{coro}{Corollary}[section]

\begin{document}
	
	\title[ Spectral stability of periodic waves.]
	{On the spectral stability of periodic traveling waves for the
		critical Korteweg-de Vries and Gardner equations}

	
	\thanks{{\it Date}: April, 2021.}
	
	\maketitle

	\begin{center}
		{\bf F\'abio Natali}
		
		{\scriptsize    Department of Mathematics, State University of
			Maring\'a, Maring\'a, PR, Brazil. }\\
		{\scriptsize fmanatali@uem.br }

		\vspace{3mm}
		
		{\bf Eleomar Cardoso Jr.}
		
		{\scriptsize   Federal University of Santa Catarina, Blumenau, SC, Brazil.}\\
		{\scriptsize eleomar.junior@ufsc.br}

		\vspace{3mm}
		
		{\bf Sabrina Amaral}
		
		{\scriptsize    Department of Mathematics, State University of
			Maring\'a, Maring\'a, PR, Brazil. }\\
		{\scriptsize sabrinasuelen20@gmail.com}
		
	\end{center}

	\begin{abstract}
		In this paper, we determine spectral
		stability results of periodic waves for the critical Korteweg-de Vries and Gardner equations. For the first equation, we show that both positive and zero mean periodic traveling wave solutions possess a threshold value which may provides us a rupture in the spectral stability. Concerning the second equation, we establish the existence of periodic waves using a Galilean transformation on the periodic cnoidal solution for the modified Korteweg-de Vries equation and for both equations, the threshold values are the same. The main advantage presented in our paper concerns in solving some auxiliary initial value problems to obtain the spectral stability.
	\end{abstract}
	
	\section{Introduction}
	
	Spectral stability of periodic waves is a subject of research
	in which several mathematicians have been interested in the last years. In
	this paper, we establish this property for the critical Korteweg-de Vries (critical KdV henceforth) equation
	\begin{equation}\label{4kdv}
		u_t+(u^5)_x+u_{xxx}=0,
	\end{equation}
	and the Gardner equation
	\begin{equation}\label{gardner1}
		u_t+(u^2)_x+(u^3)_x+u_{xxx}=0.
	\end{equation}
	\indent Both equations are examples of the generalized Korteweg-de Vries equation (gKdV henceforth), that is,
	\be\label{gkdv} u_t+(g(u)+u_{xx})_x=0,\ \ \ \ \ \ \ \
	(x,t)\in\mathbb{R}\times\mathbb{R}, \ee
	where $g:\mathbb{R}\rightarrow\mathbb{R}$ is a smooth real function. This model is very important in several physical settings. For instance, it appears in shallow water surfaces, in internal water waves, in nonlinear optics, in the
	evolution of ion-acoustic waves, in unmagnetized plasma, and in
	nonlinear hydromagnetic waves in a cold collisionless plasma
	(\cite{be1}, \cite{be2}, \cite{Dodd}, \cite{kdv}).\\
	\indent The general equation above admits traveling waves of
	the form $u(x,t)=\phi(x-\omega t)$, where $\omega$ indicates the
	wave-speed. By substituting this type of solution into $(\ref{gkdv})$
	and integrating it once, we obtain
	\begin{equation}\label{travkdv}
		-\omega\phi+\phi''+g(\phi)+A=0,
	\end{equation} where $A$ is an arbitrary integration constant and
	$\phi=\phi_{(\omega,A)}$ has $L-$periodic boundary conditions.\\
	\indent Now, we present the basic framework about spectral stability, as well as some important contributions concerning this subject. The problem of spectral instability for the gKdV equation
	consists in proving the existence of $\lambda\in\mathbb{C}$ with
	$\mbox{Re}(\lambda)>0$ such that its corresponding eigenfunction
	$u\neq0$ satisfies the linearized equation \be\label{linprob}
	\partial_x\mathcal{L}u=\lambda u.
	\ee Here,
	$\mathcal{L}$ denotes the linearized operator around the traveling
	wave $\phi$ defined in $L^2$ given by \be\label{linop}
	\mathcal{L}=-\partial_x^2+\omega-g'(\phi),\ee
	where $g'$ indicates the derivative of $g$ in terms of $\phi$. In the affirmative case, $\phi$ is said to be spectrally unstable. Otherwise, the periodic wave $\phi$ is said to be spectrally stable
	if the spectrum of $\partial_x\mathcal{L}$ is entirely  contained in the imaginary axis of the complex plane $\mathbb{C}$.  \\
	\indent If we restrict to the case of solitary
	waves, $J=\partial_x$ is a one-to-one operator with no  bounded inverse. This fact prevents the use of classical methods of spectral instability as in \cite{grillakis1}. However,
	sufficient conditions have been established by some contributors to overcome this difficulty. For instance,
	in \cite{kap} the authors determined results of spectral
	stability related to the problem $(\ref{linprob})$ by using the
	Krein-Hamiltonian instability index. Moreover, it is possible to adapt the method to
	conclude similar facts for the BBM-type problems \be\label{BBM}
	u_t+u_x-u_{txx}+(g(u))_x=0,\ee and the fractional models
	related to those equations. In \cite{lopes} the author presented
	sufficient conditions for the linear instability by
	using the semigroup theory. Interesting results were also given by \cite{A} and \cite{lin}.\\
	\indent In periodic setting we have the work \cite{DK}, where sufficient conditions for the spectral stability/instability have been determined.
	However in such case, since $J$ is not a one-to-one operator, the authors have
	considered the modified problem \be\label{modspecp1}
	J\mathcal{L}\big|_{H_0}u=\lambda u, \ee where $H_0\subset
	H:=L_{per}^2([0,L])$ is the closed subspace given by
	$$H_0=\left\{f\in L^2([0,L]);\ \int_0^Lf(x)dx=0\right\}.$$ The
	Krein-Hamiltonian index formula was applied to deduce the spectral
	stability of periodic waves for the equation
	$(\ref{gkdv})$ with $g(s)=\pm s^3$. However, it was necessary to know the behavior of the first
	five eigenvalues of the linear operator $\mathcal{L}$ in
	$(\ref{linop})$. \\
	\indent In our analysis, the previous knowledge of the periodic wave
	is not necessary but it can be useful in order to obtain the spectral stability/instability. We use a different way to compute the Krein-Hamiltonian index formula for some specific examples but the method can be adapted to other models contained in the regime of $(\ref{gkdv})$ and related equations. Presented here are the critical KdV and Gardner equations.\\
	\indent We are going to illustrate our two basic examples. First, if the critical KdV is considered, we prove the spectral stability/instability results associated with the positive and zero mean periodic waves. It is well known that both periodic waves appear when, in the equation $(\ref{travkdv})$, $g(\phi)=\phi^5$ and $A=0$. Concerning positive and periodic waves, we determine our results  using the explicit solution determined in \cite{AN2}. After that, we solve numerically some auxiliary initial value problems which give us the precise information about the Krein-Hamiltonian index formula in order to obtain the spectral stability. Explicit zero mean periodic waves were unknown in the current literature until now. To fill this gap, we present a cnoidal wave profile.\\
	\indent For the case of positive periodic waves (see ($\ref{dn4kdv}$)), it is expected that $\ker(\mathcal{L})=[\phi']$ and $n(\mathcal{L})=1$, where $n(\mathcal{L})$ indicates the number of negative eigenvalues of $\mathcal{L}$. When periodic waves with the zero mean property are considered (see $(\ref{cnoid4kdv})$), we have $n(\mathcal{L})=2$ and $\ker(\mathcal{L})=[\phi']$. In the second case, and using an explicit solution, it has been determined the same spectral property for the case $g(s)= s^3$ and $A=0$ as determined \cite{AN1} and \cite{DK}. It is worth mentioning that the authors had in hands the behavior of the first five eigenvalues of the linearized operator $\mathcal{L}$ to calculate the sign of $\langle\mathcal{L}^{-1}1,1\rangle$ using Fourier series. This quantity plays an important role to deduce spectral stability results for the gKdV equation in the sense that it is possible to identify an eventual existence of points $(\omega,A)$ contained in the parameter regime satisfying $\langle\mathcal{L}^{-1}1,1\rangle<0$ and $\langle\mathcal{L}^{-1}1,1\rangle>0$. The change of sign establishes a rupture on the spectral stability scenario when cnoidal waves for the modified KdV equation are considered (\cite{AN1} and \cite{DK}). Concerning our zero mean periodic waves for the critical KdV we obtain, in the line $(\omega,0)$, a threshold value $\omega_1>0$ such that $\langle\mathcal{L}^{-1}1,1\rangle=0$. In addition, $\langle\mathcal{L}^{-1}1,1\rangle<0$, if $\omega<\omega_1$ and $\langle\mathcal{L}^{-1}1,1\rangle>0$, if $\omega>\omega_1$. In the first case, the wave is spectrally stable and in the second one, spectrally unstable. For dnoidal waves, there is no threshold value $\omega_1$ for the quantity $\langle\mathcal{L}^{-1}1,1\rangle$ and the periodic waves are spectrally stable. Summarizing our results, we have the following theorem:
	\begin{teo}\label{teoest} Let $L>0$ be fixed.\\
		a) For all $\omega>\frac{\pi^2}{L^2}$, positive and periodic waves of dnoidal type for the critical KdV equation are spectrally stable.\\
		b) There exists a unique $\omega_1>\frac{4\pi^2}{L^2}$ such that the zero mean periodic waves of cnoidal type for the critical KdV equation are spectrally stable for $\omega\in \left(\frac{4\pi^2}{L^2},\omega_1\right)$ and spectrally unstable for $\omega>\omega_1$.
	\end{teo}
\begin{obs}
	  Theorem $\ref{teoest}$-a) establishes the spectral stability of the periodic dnoidal waves. This solution first appeared in \cite{AN2} and the authors established the existence of a unique $\omega_0>\frac{\pi^2}{L^2}$ such that the dnoidal wave is orbitally stable for $\omega\in \left(\frac{\pi^2}{L^2},\omega_0\right)$ (using the classical argument in \cite{grillakis1}) and orbitally unstable for $\omega>\omega_0$ (employing an adaptation of the arguments in \cite{bona}). Since the Cauchy problem for the equation $(\ref{4kdv})$ is not globally well posed in the energy space $H_{per}^1([0,L])$, we are in conformity with the arguments in \cite{AN2}. In fact, we are attesting for KdV type equations that spectral stability implies the orbital stability provided that the global well posed in the energy space $H_{per}^1([0,L])$ of the associated Cauchy problem is verified.
\end{obs}

	\indent Next, we shall give few words about the Gardner equation. We construct explicit periodic waves with cnoidal profile by using the modified KdV equation and its corresponding cnoidal solution. In fact, if $\phi$ is a solution of the equation $(\ref{travkdv})$ with $g(s)=s^2+s^3$, thus $\varphi=\phi+\frac{1}{3}$ is a periodic solution with cnoidal profile as $\varphi(x)=d{\rm cn}(ex,k)$ and for the corresponding modified KdV equation
	\begin{equation}\label{mkdv12}
	-\varphi''+\left(\omega+\frac{1}{3}\right)\varphi-\varphi^3=0,
	\end{equation}
	where $d$ and $e$ are smooth functions depending on the wave speed $\omega+\frac{1}{3}$. In equation $(\ref{mkdv12})$, $k\in(0,1)$ is called modulus of the elliptic function.\\
	\indent It is well known that equation $(\ref{mkdv12})$ admits periodic waves with dnoidal and cnoidal profiles. The corresponding solution with dnoidal profile for the Gardner equation and its respective orbital stability have been determined in \cite{AP2}. Our intention is to determine spectral stability results of the associated cnoidal profile and we also present a threshold value $\omega_1$ such that $\langle\mathcal{L}^{-1}1,1\rangle=0$ at $\omega=\omega_1$. More specifically, we obtain the same threshold value $\omega_1$ as obtained for the cnoidal waves for the modified KdV equation and the reason for that concerns a connection between modified KdV and Gardner equations using the Galilean invariance $\varphi=\phi+\frac{1}{3}$. This fact produces that the linearized operator associated to both periodic waves $\varphi$ and $\phi$ are the same. As a consequence, if $\mathcal{L}_{\varphi}$ is the corresponding linearized operator around $\varphi$ for the modified KdV equation and $\mathcal{L}$ the linearized operator around $\phi$, we have $\langle\mathcal{L}_{\varphi}^{-1}1,1\rangle=\langle\mathcal{L}^{-1}1,1\rangle$, that is, the sign of $\langle\mathcal{L}^{-1}1,1\rangle$ is determined just by analysing $\langle\mathcal{L}_{\varphi}^{-1}1,1\rangle$. Thus, $\omega<\omega_1$  implies that $\langle\mathcal{L}^{-1}1,1\rangle<0$ (spectral stability) while $\omega>\omega_1$ gives us $\langle\mathcal{L}^{-1}1,1\rangle>0$ (spectral instability). Summarizing our results, we have:
	\begin{teo}\label{teoestG} Let $L>0$ be fixed. There exists a unique $\omega_2>\frac{4\pi^2}{L^2}$ such that the zero mean periodic waves of cnoidal type for the Gardner equation are spectrally stable for $\omega\in \left(\frac{4\pi^2}{L^2},\omega_2\right)$ and spectrally unstable for $\omega>\omega_2$.
	\end{teo}
	
	\indent This paper is organized as follows. In Section 2 we give the
	basic framework of the spectral stability following the ideas in \cite{DK}. In Section 3, we study the existence of periodic waves and
	their dependence with respect to the parameters, as well as the
	spectral analysis of the linearized operator. Finally, Section 4 is devoted to our applications.

	\section{Basic Framework of Spectral Stability of Periodic Waves.}
	
	\setcounter{equation}{0}
	\setcounter{defi}{0}
	\setcounter{teo}{0}
	\setcounter{lema}{0}
	\setcounter{prop}{0}
	\setcounter{coro}{0}

	In this section we present the basic framework established in
	\cite{DK} which provides us a criterion for determining the spectral
	stability of periodic waves related to the abstract Hamiltonian
	equations of the form
	
	\be\label{Hamilt} u_t=J\mathcal{E}'(u)\ee defined on a Hilbert space
	$H$, where $J:H\to {\rm{range}}(H)\subset H$ is a skew symmetric, and
	$\mathcal{E}:H\to\mathbb{R}$ is a $C^2-$functional. We restrict
	ourselves to the specific case when $J=\partial_x$ and
	$\mathcal{E}(u)=\frac{1}{2}\int_0^Lu_x^2-2G(u)dx$,
	where $G'=g$. In that case, the equation $(\ref{Hamilt})$ becomes the well known
	generalized Korteweg-de Vries equation as in $(\ref{gkdv})$.\\
	\indent Let us consider again the spectral problem related to the
	generalized KdV equation
	
	\be\label{specp}
	\partial_x\mathcal{L}u=\lambda u,
	\ee where
	$\mathcal{L}=-\partial_x^2+\omega-g'(\omega,A,\phi)$ is the
	linearized operator around the periodic wave $\phi$ which is
	a periodic traveling wave solution of the equation $(\ref{travkdv})$.  As we have mentioned before, the standard theories
	of spectral instability of traveling waves for the abstract
	Hamiltonian system as in \cite{grillakis1} and
	\cite{lopes} can not be applied in this context. To overcome this
	difficulty, we are going to give a brief explanation of the results in \cite{DK}.
	Indeed, let us consider the modified spectral problem obtained from
	$(\ref{specp})$
	
	\be\label{modspecp} J\mathcal{L}\big|_{H_0}u=\lambda u,
	\ee
	where $H_0\subset H=L_{per}^2([0,L])$ is the closed subspace given by
	
	$$H_0=\left\{f\in L^2([0,L]);\ \int_0^Lf(x)dx=0\right\}.$$
	For a fixed period $L>0$, we need to assume in this whole section that:\\
	
	{\rm (a1)} There exists a fixed pair $(\omega_0,A_0)$ and $\phi:=\phi_{(\omega_0,A_0)}$ smooth even periodic solution for the equation $(\ref{travkdv})$. Moreover, we assume that
	$\phi'$ has only two zeros in the interval $[0,L)$. \\
	
	{\rm (a2)} $\ker(\mathcal{L})=[\phi']$.\\

	\indent Assumption {\rm (a1)} implies, from the classical Floquet theory in \cite{Magnus} that $n(\mathcal{L})=1$ or $n(\mathcal{L})=2$ where
	$n(\mathcal{L})$ indicates the number of negative eigenvalues of the linearized operator $\mathcal{L}=-\partial_x^2+\omega-g'(\phi)$. In addition, assumption {\rm (a2)} allows us to deduce the existence of a non-periodic even solution $\bar{y}$ which satisfies the Hill equation
	\be\label{Hill1}
	-\bar{y}''+\omega\bar{y} -g'(\phi) \bar{y}=0,
	\ee
	where $\{\bar{y},\phi'\}$ is a fundamental set of solutions for the linear equation $(\ref{Hill1})$.\\
	\indent According with Theorem \ref{teo2} determined in Section 3, one can see that assumption {\rm (a2)} will provide us the existence of a smooth surface of even periodic waves which solves $(\ref{travkdv})$ and defined in an open subset $\mathcal{O}\subset\mathbb{R}^2$,
	$$(\omega,A)\in \mathcal{O}\mapsto\phi_{(\omega,A)}\in H_{per}^s([0,L]),\ \ s\gg1,$$ all of them with the same period $L>0$. In what follows and in the whole paper, we shall not distinguish the periodic wave $\phi$ for a fixed pair $(\omega_0,A_0)$ and $\phi$ for a pair $(\omega,A)\in\mathcal{O}$ since both have the \textit{same fixed period $L>0$}. The intention is to simplify our presentation with easier notations.\\
	\indent Next, we describe the arguments in \cite{DK}. For the spectral problem in $(\ref{modspecp})$ let $k_r$ be the number of real-valued and positive eigenvalues (counting multiplicities). The quantity $k_c$ denotes the number of complex-valued eigenvalues with a positive real part. Since ${\rm Im}(\mathcal{L})=0$, where ${\rm Im}(z)$ indicates the imaginary part of the complex number $z$, we see that $k_c$ is an even integer. For a self-adjoint operator $\mathcal{A}$, let $n(\langle w,\mathcal{A}w\rangle)$ be the dimension of the maximal subspace for which $\langle w,\mathcal Aw\rangle<0$. Also, let $\lambda$ be an eigenvalue and $E_{\lambda}$ its corresponding eigenspace. The eigenvalue is said to have negative Krein signature if
	$$k_i^{-}(\lambda):=n(\langle w,(\mathcal{L}\big|_{H_0})\big|_{E_{\lambda}}w\rangle)\geq1,$$
	otherwise, if $k_i^{-}(\lambda)=0$, then the eigenvalue is said to have a positive Krein signature. If $\lambda$ is geometrically and algebraically simple with the eigenfunction $\psi_{\lambda}$, then
	$$k_i^{-}(\lambda)=\left\{\begin{array}{llll}
		0,\ \langle \psi_{\lambda},(\mathcal{L}\big|_{H_0})\psi_{\lambda}\rangle>0\\
		1,\ \langle \psi_{\lambda},(\mathcal{L}\big|_{H_0})\psi_{\lambda}\rangle<0.\end{array}
	\right.$$
	
	We define the total Krein signature as
	$$k_i^{-}:=\sum _{\lambda\in i\mathbb{R}\backslash\{0\}}k_{i}^{-}(\lambda).$$
	The fact ${\rm Im}(\mathcal{L})=0$ implies that $k_i^-(\lambda)=k_i^{-}(\overline{\lambda})$ and $k_i^{-}$ is an even integer.\\
	
	Let us consider
	\be\label{I}
	\mathcal{I}=\langle \mathcal{L}^{-1}1,1\rangle.
	\ee
	If $\mathcal{I}\neq0$, denote $\mathcal{D}$ as the $2\times 2-$matrix given by
	\be\label{Dmatrix}
	\mathcal{D}=\frac{1}{\langle\mathcal{L}^{-1}1,1\rangle}\left[\begin{array}{llll}
		\langle\mathcal{L}^{-1}\phi,\phi\rangle & & \langle\mathcal{L}^{-1}\phi,1\rangle\\\\
		\langle\mathcal{L}^{-1}\phi,1\rangle & & \langle\mathcal{L}^{-1}1,1\rangle\end{array}\right]
	\ee
	We obtain, then, the following results:
	\begin{teo}\label{krein}
		Suppose that assumptions {\rm (a1)-(a2)} hold. If $\mathcal{I}\neq0$ and $\mathcal{D}$ is non-singular (i.e. $\det(\mathcal{D})\neq 0$) we have for the eigenvalue problem $(\ref{modspecp})$
		$$k_r+k_c+k_{i}^{-}=n(\mathcal{L})-n(\mathcal{I})-n(\mathcal{D}).$$
		The nonpositive integer $K_{{\rm Ham}}=k_r+k_c+k_{i}^{-}$ is called \textit{Hamiltonian-Krein index}.
	\end{teo}
	\begin{proof}
		See Theorem 1 in \cite{DK}.
	\end{proof}
	
	\begin{coro}\label{coroest}
		Under the assumptions of Theorem $\ref{krein}$, if $k_c=k_r=k_{i}^{-}=0$ the periodic wave $\phi$ is spectrally stable. In addition, if $K_{{\rm Ham}}=1$ the refereed periodic wave is spectrally unstable.
	\end{coro}
	\begin{proof}
		Since $k_c=k_r=0$, there is no eigenvalues with positive real part for the problem $(\ref{modspecp})$ and $\phi$ is spectrally stable since the total Krein signature is zero. Now, if $K_{{\rm Ham}}=1$ we deduce that $k_r=1$ since $k_c$ and $k_{i}^{-}$ are even nonnegative integers. So, operator $J\mathcal{L}$ presented in the spectral problem $(\ref{modspecp})$ has a positive eigenvalue which enable us to deduce the spectral instability of the periodic wave.
	\end{proof}
	
	We shall present some considerations concerning the result determined in Corollary $\ref{coroest}$ applied to the case of the generalized KdV equation in $(\ref{gkdv})$. In fact, by assuming that assumption {\rm (a2)} is verified one has
	$$\langle\mathcal{L}^{-1}\phi,\phi\rangle=-\frac{1}{2}\frac{d}{d\omega}\int_0^L\phi^2dx,\ \ \langle\mathcal{L}^{-1}\phi,1\rangle=-\frac{d}{d\omega}\int_0^L\phi dx,$$
	and
	$$\mathcal{I}=\langle\mathcal{L}^{-1}1,1\rangle=\frac{d}{dA}\int_0^L\phi dx.$$
	So, we have
	$$n(\mathcal{I})=\left\{\begin{array}{llll}
		0,\ \frac{d}{dA}\int_0^L\phi dx\geq0\\\\
		1,\ \frac{d}{dA}\int_0^L\phi dx<0.\end{array} \right.$$
	
	On the other hand, in order to determine $n(\mathcal{D})$ it is necessary
	to analyze the quantity $\mathcal{D}$.
	In fact, if $\mathcal{D}<0$ we have that the associated matrix has a
	positive and a negative eigenvalue and therefore $n(\mathcal{D})=1$.
	However, if $\mathcal{D}>0$, it is not possible to directly decide
	about the quantity $n(\mathcal{D})$ since we could have
	$n(\mathcal{D})=0$ (two positive eigenvalues for the associated matrix) or $n(\mathcal{D})=2$
	(two negative eigenvalues). In the next section,
	we determine sufficient conditions to obtain assumptions {\rm
		(a1)-(a2)} for a general class of second order differential equations. In addition, we present two useful initial value problems used to determine a precise way to calculate $\mathcal{I}$ and $\mathcal{D}$.

	\section{Basic Framework on Spectral Analysis.}
	\setcounter{equation}{0}
	\setcounter{defi}{0}
	\setcounter{teo}{0}
	\setcounter{lema}{0}
	\setcounter{prop}{0}
	\setcounter{coro}{0}

	\indent In a general setting (without considering the arguments in the last section for a while), let us suppose that $\phi$ is an even $L-$periodic solution of the general equation
	\begin{equation}\label{ode}
		-\phi''+f(\omega,A,\phi)=0,
	\end{equation}
	where $f$ is a smooth function depending on $(\omega,A,\phi)$ and $(\omega,A)$ is an element of an admissible set $\mathcal{P}\subset\mathbb{R}^2$. This means that $\mathcal{P}$ contains all the pairs $(\omega,A)$ where $\phi$ is a periodic solution of $(\ref{ode})$.\\
	\indent Let $\mathcal{L}$ be the linearized equation around $\phi$, where $\phi$ is a periodic
	solution of (\ref{ode}) of period $L$. The linearized operator \be
	\mathcal{L}(y) = - y'' + f'(\omega, A, \phi)\, y,
	\;\;\; (\omega, A) \in \mathcal{P} \label{hill} \ee is a Hill
	operator and $f'$ is the derivative in terms of $\phi$. According to \cite{Haupt} and \cite{Magnus},
	the spectrum of $\mathcal{L}$ is formed by an unbounded
	sequence of real numbers
	\[
	\lambda_0 < \lambda_1 \leq \lambda_2 < \lambda_3 \leq \lambda_4\;\;
	...\; < \lambda_{2n-1} \leq \lambda_{2n}\; \cdots,
	\]
	where equality means that $\lambda_{2n-1} = \lambda_{2n}$  is a
	double eigenvalue. The spectrum of $\mathcal{L}$ is
	characterized by the number of zeros of the eigenfunctions, if $\Psi$
	is an eigenfunction for the eigenvalue $\lambda_{2n-1}$ or
	$\lambda_{2n}$, then $\Psi$  has exactly $2n$ zeros in the half-open
	interval $[0,L)$.
	
	In order to apply the general theory of orbital stability,
	\cite{bona}, \cite{grillakis1} and \cite{W1}, the spectrum of
	$\mathcal{L}$ is of main importance and also of the major
	difficulty in the applications. It is necessary to know exactly the
	non-positive spectrum; more precisely, it is necessary to know the
	inertial index $in(\mathcal{L})$ of
	$\mathcal{L}$, where $in(\mathcal{L})$ is
	a pair of integers $(n,z)$, where $n$ is the dimension of the
	negative subspace of $\mathcal{L}$ and $z$ is the
	dimension of the null subspace of $\mathcal{L}$.
	
	The results of this section are based on \cite{natali2}, \cite{neves} and \cite{neves1} and the first
	one concerns the invariance of the index with respect to the parameters. Since the derivative  $\phi'$ is an eigenfunction related to $\lambda =0$ for every  $(\omega, A)
	\in \mathcal{P}$, we can state the following result.
	
	\begin{teo}
		Let $\phi$ a smooth $L-$periodic solution of the equation $(\ref{ode})$.
		Then the
		family of operators $\mathcal{L}(y) = - y'' +
		f'(\omega, A, \phi)\, y $ is isoinertial with respect to $(\omega,A)$ in the parameter regime. \label{teo0}
	\end{teo}
	\begin{proof}
		See \cite{natali2} and \cite{neves1}.
	\end{proof}
	
	In order to calculate the inertial index of $\mathcal{L}$ for a fixed value of $(\omega_0,A_0)$, we shall consider the
	auxiliary function $\bar{y}$ the unique solution of the problem \be
	\left\{
	\begin{array}{l}
		- \bar{y}'' + f'(\omega_0, A_0, \phi) \bar{y} = 0 \\
		\bar{y}(0) = - \frac{1}{\phi''(0)} \\
		\bar{y}'(0)=0,
	\end{array} \right.
	\label{y} \ee and also the constant  $\theta$ given by \be \theta=
	\frac{ \bar{y}'(L)}{\phi''(0)}, \label{theta} \ee where $L$ is
	the period of $\phi=\phi_{(\omega_0,A_0)}$.
	
	We know that the derivative $\phi'$ is an eigenfunction
	for  the eigenvalue $\lambda = 0$, and also that $\phi'(x)$
	has exactly two zeros in the half-open interval $[0, L)$.
	Therefore we have three possibilities:   \begin{itemize}
		\item[i)] $\lambda_1 = \lambda_2 = 0 \Rightarrow in(\mathcal{L}) = (1,2)$,\\
		\item[ii)] $\lambda_1 = 0 < \lambda_2 \Rightarrow in(\mathcal{L}) = (1,1)$,\\
		\item[iii)] $\lambda_1 < \lambda_2 = 0 \Rightarrow  in(\mathcal{L}) = (2,1)$,
	\end{itemize}
	
	The method we use to decide and calculate the inertial index is
	based on Lemma 2.1 and Theorems 2.2 and 3.1 of \cite{neves}. This
	result can be stated as follows.

	\begin{teo}
		Let $\theta$ be the constant given by (\ref{theta}), then the
		eigenvalue $\lambda=0$ is simple if and only if $ \theta \neq 0$.
		Moreover, if $\theta \neq 0$, then  $ \lambda_{1}=0$ if $\theta <
		0$, and $ \lambda_{2}=0$ if $\theta > 0$. \label{teo1}
	\end{teo}
	\begin{flushright}
		$\square$
	\end{flushright}

	Let $L>0$ be fixed. In order to show our spectral stability results, it is convenient to show the existence
	of a family $\phi$ of $L$-periodic solutions for the
	equation (\ref{ode}) that smoothly depends on the parameters
	$(\omega,A)$, for $(\omega,A)$  in an open set $\mathcal{O} \subset
	\mathcal{P}$.

	\begin{teo}
		Let $\phi_{(\omega_0,A_0)}$ be an
		even periodic solution of $(\ref{ode})$ defined in a fixed pair $(\omega_0,A_0)$ in the parameter regime. If $\theta \neq 0$, where
		$\theta$ is  the constant given in Theorem \ref{teo1}, and $L$ is
		the period of $\phi_{(\omega_0,A_0)}$, then there is an open
		neighborhood $\mathcal{O}$ of $(\omega_0,A_0)$,
		and a family $\phi_{(\omega,A)} \in H_{per,e}^2([0,L])$ of
		$L$-periodic solutions of $(\ref{ode})$, which smoothly depends on
		$(\omega,A) \in \mathcal{O}$ in a $C^1$ manner. \label{teo2}
	\end{teo}
	\begin{proof}
		Let $\mathcal{P}$ the set of parameters and $\mathcal{F}$ be the operator given by the equation (\ref{ode}) restrict
		to the even functions, precisely, $\mathcal{F}:\mathcal{P}\times
		H_{per,e}^2([0,L])  \rightarrow  L_{per,e}^2([0,L]) $, \be
		\mathcal{F}(\omega,A,\phi) = -\phi''+ f(\omega,A,\phi). \label{eq31}
		\ee Then $\mathcal{F}(\omega_0,A_0,\phi_{(\omega_0,A_0)}) = 0$,
		since $\phi_{(\omega_0,A_0)}$ is an even periodic solution of the
		equation (\ref{ode}). If $\theta \neq 0 $, Theorem \ref{teo1}
		implies that $ \mathcal{L}_{(\omega_0, A_0)}(y) = - y'' +
		f'(\omega_0, A_0, \phi_{(\omega_0,A_0)})\, y$, has an one-dimensional nullspace;
		and from the invariance, this nullspace is spanned by
		$\phi'_{(\omega_0, A_0)}$. Since $\phi'_{(\omega_0, A_0)}$ is odd,
		it is not an element of $H_{per,e}^2([0,L])$, it follows that $
		\mathcal{F}_{\phi}(\omega_0,A_0,\phi_{(\omega_0,A_0)}) =
		\mathcal{L}_{(\omega_0, A_0)}:H_{per,e}^2([0,L]) \subset L_{per,e}^2([0,L])\rightarrow
		L_{per,e}^2([0,L]) $ is invertible and its inverse is bounded.
		Therefore, the results of the Theorem \ref{teo2} follows from the
		implicit function theorem. See Theorem 15.1 and Corollary 15.1 of
		\cite{Deimling}.
	\end{proof}
	
	Next, we turn back to the setting contained in Section 2 by considering $(\ref{ode})$ as
	
	\begin{equation}\label{odekdv}
		-\phi''+\omega\phi-g(\phi)-A=0.
	\end{equation}
	We assume that $\theta\neq0$ in a single point $(\omega_0,A_0)$  in the parameter regime. By Theorem $\ref{teo2}$ we can define
	\[
	\psi = \frac{\partial \phi}{\partial \omega} \qquad
	\mbox{and} \qquad \eta= \frac{\partial \phi}{\partial
		A}.
	\]

	Again by Theorem \ref{teo2}, it is easy to see that $\psi$
	above is an even periodic smooth function which satisfies, for the case of the equation $(\ref{gkdv})$
	\be -
	\psi'' + \omega \psi -g'(\phi)\psi= -
	\phi. \label{psi1} \ee In addition, $\eta$ is also an
	even periodic function satisfying
	
	\be - \eta'' + \omega \eta
	-g'(\phi)\eta= 1. \label{eta1} \ee
	
	\begin{obs}\label{obsiso}
		Theorem $\ref{teo0}$ gives us an important property concerning the quantity and multiplicity of the first two eigenvalues associated to the linearized operator $\mathcal{L}$ defined in $(\ref{linop})$. Indeed, if $\theta\neq0$ in a certain point $(\omega_0,A_0)$ in the parameter regime $\mathcal{P}$, we can conclude that the kernel of $\mathcal{L}$ is simple and $n(\mathcal{L})$ is constant for all $(\omega,A)$ in an open subset contained in $\mathcal{P}$, that is, the value $in(\mathcal{L})$ is constant in this subset.
		
	\end{obs}
	
	Next result gives us an immediate converse of Theorem $\ref{teo2}$ for the case $g(s)=s^{p+1}$.
	\begin{prop}\label{propsimp}
		Let $\widetilde{\mathcal{O}}\subset\mathbb{R}^2$ be an open subset. Suppose that $(\omega,A)\in\widetilde{\mathcal{O}}\mapsto\phi_{(\omega,A)}$ is a smooth surface of even (odd) periodic traveling wave solutions which solves $(\ref{odekdv})$ with $g(s)=s^{p+1}$  all of them with the same fixed period $L>0$. Then, $\ker(\mathcal{L})=[\phi']$ and the value $n(\mathcal{L})$ is constant for all $(\omega,A)\in\widetilde{\mathcal{O}}$. The same result remains valid for the case $A\equiv0$, by considering $\widetilde{I}\subset\mathbb{R}$ an open subset and $\omega\in \widetilde{I}\mapsto\phi_{\omega}$ a smooth curve of even periodic waves.
	\end{prop}
	\begin{proof}
		To simplify the notation, let us denote $\phi=\phi_{(\omega,A)}$ and consider $\{\phi',\bar{y}\}$ the fundamental set of solutions related to the equation $-y''+\omega y - (p+1)y^p=0$. By contradiction, assume that $\bar{y}$ is $L-$periodic. Since $\phi'$ is odd, the arguments in \cite{Magnus} give us that $\bar{y}$ can be considered even. The Wronskian of the set $\{\phi',\bar{y}\}$ and denoted by $\mathcal{W}(\phi',\bar{y})$ satisfies $\mathcal{W}(\phi',\bar{y})=1$ over $[0,L]$ (see \cite{Magnus}). Moreover, since $\bar{y}$ and $\phi'$ are both periodic functions, we obtain from $(\ref{travkdv})$ that
		\begin{equation}\label{wronsk}\begin{array}{lllll}
				L&=&\displaystyle\int_0^{L}\mathcal{W}(\phi',\bar{y})dx=-2\int_0^{L} \bar{y}\phi''dx=-2\int_0^{L}\bar{y}\left(\omega\phi-\phi^{p+1}-A\right)dx\\\\
				&=&\displaystyle-2\omega\int_0^{L} \bar{y}\phi dx+2\int_0^{L} \bar{y}\phi^{p+1}dx+2A\int_0^{L}\bar{y}dx.\end{array}
		\end{equation}
		Since $\mathcal{L}\phi=A-p\phi^{p+1}$, by $(\ref{psi1})$ and $(\ref{eta1})$ we obtain from $(\ref{wronsk})$
			\begin{equation}\label{wronsk1}
				L=2\omega \langle \mathcal{L}\psi,\bar{y}\rangle-\frac{2}{p}\langle\mathcal{L}\phi,\bar{y}\rangle+2A\left(\frac{1}{p}+1\right)\langle\mathcal{L}\eta,\bar{y}\rangle.
			\end{equation}
\indent The fact that $\mathcal{L}\bar{y}=0$ allows us to deduce from the self-adjointness of $\mathcal{L}$ and $(\ref{wronsk1})$ that $L=0$. This contradiction shows that $\ker(\mathcal{L})=[\phi']$.
	\end{proof}

	Let us suppose that $\theta\neq0$ in a single point $(\omega_0,A_0)$ in the parameter regime. By Theorem $\ref{teo2}$ we are able to determine the initial condition
	$\psi(0)$ at the point $(\omega_0,A_0)$. To do so, we multiply equation (\ref{psi1}) by $\bar{y}$, where
	$\bar{y}$ is given in (\ref{y}), and integrate the
	first term twice. We get
	$$
	- \int_0^{L}  \phi_{(\omega_0,A_0)} \;
	\bar{y}\; dx = \psi(L) \bar{y}'(L)=\psi(0) \bar{y}'(L).
	$$
	Similarly, from $(\ref{eta1})$ one has
	$$
	\int_0^{L}\bar{y}\; dx =  \eta(L) \bar{y}'(L)=\eta(0) \bar{y}'(L).
	$$

	Since $\theta \neq 0$ we conclude that $\bar{y}'(L) \neq 0$ and then
	$\psi(x)$ and $\eta(x)$ are obtained by solving, respectively, the
	following initial value problems \be \left\{
	\begin{array}{lllllllllllll}
		- \psi'' + \omega_0 \psi -g'(\omega_0,A_0,\phi_{(\omega_0,A_0)})\psi= -
		\phi_{(\omega_0,A_0)}  \\
		\psi(0) = - \frac{1}{\bar{y}'(L)}  \int_0^{L}  \phi_{(\omega_0,A_0)} \; \bar{y}\; dx \\
		\psi'(0)=0,
	\end{array} \right.\
	\left\{
	\begin{array}{l}- \eta'' + \omega_0 \eta -g'(\omega_0,A_0,\phi_{(\omega_0,A_0)})\eta=
		1  \\
		\eta(0) =  \frac{1}{\bar{y}'(L)}  \int_0^{L}  \bar{y}\; dx \\
		\eta'(0)=0.
	\end{array} \right.
	\label{psi2}
	\ee
	Both initial value problems are very useful to determine $\mathcal{I}$ and $\mathcal{D}$ given in Section 2. 
	\section{Applications - Spectral stability of periodic waves}
	
	\setcounter{equation}{0}
	\setcounter{defi}{0}
	\setcounter{teo}{0}
	\setcounter{lema}{0}
	\setcounter{prop}{0}
	\setcounter{coro}{0}

	\subsection{Case $g(s)=s^{p+1}$ - Existence of periodic waves using variational methods.} Using a variational method, we establish the existence of periodic waves for the equation $(\ref{odekdv})$. The main advantage of the approach presented here is that the quantity of negative eigenvalues of $\mathcal{L}$ in $(\ref{linop})$ defined for periodic waves $\phi$ in a single point $(\omega_0,A_0)\in\mathcal{P}$ is precisely determined. Thus, in this specific case, Remark $\ref{obsiso}$ can be used to deduce the quantity and multiplicity of negative eigenvalues for all $(\omega,A)$.

	\indent Let $L>0$ be fixed. For each $\gamma>0$, we define the set
	$$Y_{\gamma}=\left\{u\in H_{per}^1([0,L]);\ \int_0^{L}u^{p+2}dx=\gamma\right\},$$
	where $p$ is an even integer. Our first goal is to find a minimizer of the constrained minimization problem
	\begin{equation}
		\label{infB}
		m=m_{\omega}=\inf_{ u\in Y_{\gamma}}\mathcal{B}_{\omega}(u),
	\end{equation}
	where for each $\omega>0$, $\mathcal{B}_{\omega}$ is given by
	\begin{equation}\label{Bfunctional}
		\mathcal{B}_{\omega}(u)=\frac{1}{2}\int_{0}^{L}u'^2+\omega u^2dx.
	\end{equation}
	We observe that $\mathcal{B}_{\omega}$ is a smooth functional on $H_{per}^1([0,L])$.
	\medskip
	
	\begin{lema}\label{minlem}
		The minimization problem \eqref{infB} has at least one nontrivial solution, that is, there exists $\phi\in Y_{\gamma}$ satisfying
		\begin{equation}\label{minBfunc}
			\mathcal{B}_{\omega}(\phi)=\inf_{ u\in Y_{\gamma}}\mathcal{B}_{\omega}(u).
		\end{equation}
	\end{lema}
	\begin{proof}
		Since $m\geq0$ and $\mathcal{B}_{\omega}$ is a smooth functional, we are enabled to consider $\{u_n\}=\{u_{n,\omega}\}$ as a minimizing sequence for \eqref{infB}, that is, a sequence in $Y_\gamma$ satisfying
		$\displaystyle \mathcal{B}_{\omega}(u_n)\rightarrow\inf_{u\in Y_{\gamma}} \mathcal{B}_{\omega}(u)=m, \ \mbox{as} \ n\rightarrow \infty.$
		
		\indent The fact that $\omega>0$ enables us to conclude $\{u_{n}\}$ as a bounded set in $H_{per}^1([0,L])$. Thus, modulus a subsequence, there exists
		$\phi=\phi_{\omega}\in H_{per}^1([0,L])$  such that
		$u_n\rightharpoonup \phi \ \mbox{weakly in} \ H_{per}^1([0,L]),  \ \  \mbox{as} \ n\rightarrow \infty.$

		Now, since the energy space $H_{per}^1([0,L])$ is compactly embedded in $L_{per}^{p+2}([0,L])\hookrightarrow L_{per}^2([0,L])$, we have for $n\rightarrow +\infty$ that
		$u_n\rightarrow \phi \ \mbox{in} \ L^{p+2}_{per}([0,L]),$
		that is, $\int_0^{L}\phi^{p+2} dx=\gamma$.
		
		Moreover, the weak lower semi-continuity of $\mathcal{B}_{\omega}$ gives us that
		$
		\mathcal{B}_{\omega}(\phi)\leq\liminf_{n\rightarrow \infty} \mathcal{B}_{\omega}(u_n)=m.
		$
		The lemma is now proved.
	\end{proof}

	By Lemma \ref{minlem} and  Lagrange's Multiplier Theorem, we guarantee the existence of $C_1$ such that
	\begin{equation}\label{lagrange}
		-\phi''+\omega\phi=C_1\phi^{p+1}.
	\end{equation}
	We note that $\phi$ is nontrivial because $\gamma>0$ and a standard rescaling argument enables us to deduce that the Lagrange Multiplier $C_1$ can be chosen as $C_1=1$. Now, let $L>0$ be fixed as before. Since the minimization problem $(\ref{infB})$ can be solved for any $\omega>0$, we guarantee by arguments of smooth dependence in terms of the parameters for standard ODE (see for instance, \cite[Chapter I, Theorem 3.3]{hale}), the existence of a convenient open interval $I$ and a smooth curve $\omega\in I\mapsto\phi\in H_{per}^n([0,L])$, $n\in\mathbb{N}$, satisfying the equation
	\begin{equation}\label{ode-wave13}
		-\phi''+\omega\phi-\phi^{p+1}=0.
	\end{equation}
	
	In this setting, the existence of a smooth curve of periodic waves depending on $\omega$ enables us to conclude by Proposition $\ref{propsimp}$ that $\ker(\mathcal{L})$ is simple. Concerning $n(\mathcal{L})$, we see that $\phi$ is a minimizer of $\mathcal{B}_{\omega}$ with one constraint. Since $\langle\mathcal{L}\phi,\phi\rangle<0$, we obtain by Courant's Min-Max Principle that $n(\mathcal{L})=1$.\\
	\indent Analysis above gives us the following sentence: for $\phi$  solution of $(\ref{ode-wave13})$ with $\omega_0>0$ and the corresponding single point $(\omega_0,0)$ in the parameter regime, we have that $n(\mathcal{L}_{(\omega_0,0)})=1$ and $\ker(\mathcal{L}_{(\omega_0,0)})=[\phi']$. Therefore, by Theorem $\ref{teo1}$ and Remark $\ref{obsiso}$ one has that $\theta\neq0$ for all $(\omega,A)$ in an open subset $\mathcal{O}\subset\mathcal{P}$. This means that the solution $\phi$ of $(\ref{odekdv})$ satisfies $n(\mathcal{L})=1$ and $\ker(\mathcal{L})=[\phi']$ for all $(\omega,A)\in \mathcal{O}$.
	
	\begin{obs}\label{remN1}
		Solutions $\phi$ which are minimizers of the problem $(\ref{minBfunc})$ are well determined by the analysis above. In fact, they are rounding the center point(s) in the phase portrait and have the homoclinic as a limit for large periods. The corresponding solution $\phi$ enjoys the same property. As we have already mentioned before, the parameter regime is the maximal set constituted of pairs $(\omega,A)$ such that all periodic waves round the center point(s) in the phase portrait. The analysis above gives us that  $n(\mathcal{L})=1$ and $\ker(\mathcal{L})=[\phi']$ in an open subset contained in the parameter regime.	
	\end{obs}

	\indent As before, let us consider a fixed $L>0$. Again, for a fixed $\gamma>0$, we define 
	$$Z_{\gamma}=\left\{u\in H_{per,odd}^1([0,L]);\ \int_0^{L}u^{p+2}dx=\gamma\right\},$$
	where $p$ is an even integer. Now, we need to find a minimizer of the constrained minimization problem
	\begin{equation}
		\label{infB1}
		r=r_{\omega}=\inf_{ u\in Z_{\gamma}}\mathcal{B}_{\omega}(u),
	\end{equation}
	where for each $\omega>0$, $\mathcal{B}_{\omega}$ is given by $(\ref{Bfunctional})$. We have the following result for the existence of odd periodic waves.
	
	\begin{lema}\label{minlem1}
		The minimization problem \eqref{infB1} has at least one odd nontrivial solution, that is, there exists $\phi\in Z_{\gamma}$ satisfying
		\begin{equation}\label{minBfunc1}
			\mathcal{B}_{\omega}(\phi)=\inf_{ u\in Z_{\gamma}}\mathcal{B}_{\omega}(u)=r.
		\end{equation}
	\end{lema}
	\begin{proof}
		The proof of this result is similar to the proof of Lemma $\ref{minlem}$.
	\end{proof}
	
	As before, by Lemma \ref{minlem1} and  Lagrange's Multiplier Theorem, we guarantee the existence of $C_2$ such that
	\begin{equation}\label{lagrange1}
		-\phi''+\omega\phi=C_2\phi^{p+1}.
	\end{equation}
	Since $\phi$ is nontrivial, a standard rescaling argument gives us that the Lagrange Multiplier  can be chosen as $C_2=1$. By using similar arguments as determined above, we guarantee the existence of a convenient open interval $I$ and a smooth curve $\omega\in I\mapsto\phi\in H_{per}^n([0,L])$, $n\in\mathbb{N}$, satisfying the equation
	\begin{equation}\label{ode-wave134}
		-\phi''+\omega\phi-\phi^{p+1}=0.
	\end{equation}
	
	\begin{lema}\label{lemnLodd} Let $\phi$ be the solution obtained by Lemma $\ref{minlem1}$. We have that $\ker(\mathcal{L})=[\phi']$ and $n(\mathcal{L})=2$.
	\end{lema}
	\begin{proof} The existence of a smooth curve of odd periodic waves depending on $\omega$ gives us by Proposition $\ref{propsimp}$ that $\ker(\mathcal{L})$ is simple. \\
		\indent We determine $n(\mathcal{L})$. In fact, we see that $\phi$ is a minimizer of $\mathcal{B}_{\omega}$ with one constraint in the Sobolev space $H_{per,odd}^1([0,L])$ constituted by odd functions. Since $\langle\mathcal{L}|_{odd}\phi,\phi\rangle<0$, we obtain by Courant's Min-Max Principle that $n(\mathcal{L}|_{odd})=1$. Next, solution $\phi'$ is even having two zeros in the interval $[0,L)$. Using the standard Floquet theory in \cite{Magnus}, we see that zero is not the first eigenvalue of $\mathcal{L}|_{even}$, so that $n(\mathcal{L}|_{even})\geq1$. Again from the Floquet theory, since $\phi'$ has only two zeros in the interval $[0,L)$, we see that $n(\mathcal{L})\leq2$. Therefore, the only possibility is that $n(\mathcal{L})=n(\mathcal{L}|_{odd})+n(\mathcal{L}|_{even})=2$.
	\end{proof}
	
	\indent Analysis above gives us the following sentence: for $\phi$  solution of $(\ref{ode-wave13})$ with $\omega_0>0$ and the corresponding single point $(\omega_0,0)$ in the parameter regime, we have that $n(\mathcal{L}_{(\omega_0,0)})=2$ and $\ker(\mathcal{L}_{(\omega_0,0)})=[\phi']$. Therefore, by Theorem $\ref{teo1}$ and Remark $\ref{obsiso}$ one has that $\theta\neq0$ for all $(\omega,A)$ in an open subset $\mathcal{O}\subset\mathcal{P}$. This means that the solution $\phi$ of $(\ref{odekdv})$ satisfies $n(\mathcal{L})=2$ and $\ker(\mathcal{L})=[\phi']$ for all $(\omega,A)\in \mathcal{O}$.
	\begin{obs} Defining $\varphi:=\phi(\cdot-L/4)$, we can consider the solution obtained by Lemma $\ref{minlem1}$ as being even and satisfying the mean zero condition $\int_0^{L}\varphi(x)dx=0$. In our paper and in order to avoid dubiety of notation, we keep the notation $\phi$ instead of $\varphi$ to indicate an even zero mean periodic wave satisfying equation $(\ref{ode-wave134})$. 
	\end{obs}
	\subsection{Positive periodic waves for the critical KdV - Proof of Theorem $\ref{teoest}$-a)} We start our examples studying the spectral stability of periodic waves  for $(\ref{odekdv})$ with $A=0$ and $g(s)=s^5$. According with \cite{AN2}, it is possible to determine a positive periodic wave with dnoidal profile as
	\begin{equation}\label{dn4kdv}
		\phi(x)=\frac{a{\rm dn}\left(\frac{2K(k)}{L}x,k\right)}{\sqrt{1-b{\rm sn}^2\left(\frac{2K(k)}{L}x,k\right)}},
	\end{equation}
	where $K(k)=\int_0^1\frac{dt}{\sqrt{(1-t^2)(1-k^2t^2)}}$ is the complete elliptic integral of the first kind. Parameters $a$, $b$ in $(\ref{dn4kdv})$ and the wave speed $\omega$ in $(\ref{ode-wave134})$ depend smoothly on the modulus $k\in(0,1)$ and they are given by
	\begin{equation}\label{apos}		
		a=\frac{[4(2k^2-1+2(1-k^2+k^4)^{1/2})K(k)^2L^2]^{1/4}}{L},
		\ \ \ 
		b=1-k^2-\sqrt{k^4-k^2+1},
	\end{equation}
	and
	\begin{equation}\label{wkpos}
		\omega=\frac{4K(k)^2\sqrt{k^4-k^2+1}}{L^2}.
	\end{equation}
	\indent By Remark $\ref{remN1}$, one has that $n(\mathcal{L})=1$ and $\ker(\mathcal{L})=[\phi']$. Therefore, by Theorem $\ref{teo2}$ we guarantee the existence of a smooth 
	surface $(\omega,A)\in\mathcal{O}\mapsto\phi_{(\omega,A)}\in
	H_{per,e}^n([0,L]),\ n\gg1,$ of periodic waves, all of them
	with the same period $L>0$.\\
	\indent Let $L>0$ be fixed. According with the tables below (see also Figure 1), we obtain that $\mathcal{I}>0$ and  $\det(\mathcal{D})<0$ for all $k\in(0,1)$. Then, one has the spectral stability of the periodic wave $\phi$ in an open neighbourhood of $(\omega,0)$ where $\omega>\frac{\pi^2}{L^2}$ (see Corollary $\ref{coroest}$).
	
	\begin{obs}
		Let $L>0$ be fixed. Since $\omega$ in $(\ref{wkpos})$ is a strictly increasing function depending smoothly on the modulus $k\in(0,1)$, we see for $k\rightarrow 0^+$ and the fact $K(0)=\frac{\pi}{2}$ that $\omega\rightarrow \frac{\pi^2}{L^2}^{+}$. Therefore, we have the basic estimate $\omega>\frac{\pi^2}{L^2}$ for the existence of periodic waves with dnoidal profile in $(\ref{dn4kdv})$. In addition, if $\omega>\frac{\pi^2}{L^2}$, the standard ODE theory enables us to conclude that $\phi$ in $(\ref{dn4kdv})$ is the unique positive solution which solves equation $(\ref{odekdv})$ with $A=0$ and $g(s)=s^5$. Therefore, $\phi$ in $(\ref{dn4kdv})$ satisfies the minimization problem $(\ref{infB})$
	\end{obs}

	\begin{obs} Another important fact: for a fixed $L>0$, we see that $\det(\mathcal{D})$ goes to zero when $k\rightarrow 1^{-}$ (in our tables, this fact also occurs for different values of $L$ by taking $k$ closer to $1$). Thus, we ``recover" the property $\frac{d}{d\omega}\int_{-\infty}^{\infty}Q^2dx=0$, where $Q$ is the hyperbolic secant profile for the critical KdV equation with wave speed $\omega>0$.
	\end{obs}
	\begin{table}[h]
		\centering
		\begin{tabular}{|c|c|c|}\hline
			\multicolumn{3}{|c|}{$L= 2\pi$} \\\hline
			$\kappa_0$ &  $\mathcal{I}$ & $\det(\mathcal{D})$ \\\hline
			$0.1$  & $ 5.4972$ &$-0.2692$ \\\hline
			$ 0.3$  & $5.4932$ & $-0.2690$ \\\hline
			$ 0.5$ & $ 5.4568$ & $-0.2669$\\\hline
			$0.7$  & $5.2936$ & $-0.2573$\\\hline
			$ 0.9$  &  $4.6183 $& $-0.2148$\\\hline
			$ 0.9999$ &  $1.4287$& $-0.0269$\\\hline
		\end{tabular}
		\centering
		\begin{tabular}{|c|c|c|}\hline
		\multicolumn{3}{|c|}{$L= 20$} \\\hline
		$\kappa_0$ & $\mathcal{I}$ & $\det(\mathcal{D})$ \\\hline
		$0.1$ & $177.272 $ & $-2.7286$\\\hline
		$ 0.3$  & $177.167$& $-2.7257$ \\\hline
		$ 0.5$ & $175.993$ & $-2.7046$\\\hline
		$0.7$  &$170.729$ & $-2.6074$\\\hline
		$0.9$  &  $148.949$ & $-2.1768$\\\hline
		$0.9999$  &  $46.0816$ & $-0.2772$\\\hline
		\end{tabular}
		\end{table}
	\begin{table}[h]
		\centering
		\begin{tabular}{|c|c|c|}\hline
			\multicolumn{3}{|c|}{$L= 50$} \\\hline
			$\kappa_0$ & $\mathcal{I}$ & $\det(\mathcal{D})$ \\\hline
			$0.1$  & $2770.18$ & $-17.0531$\\\hline
			$ 0.3$  & $2768.24$ & $-17.0361$\\\hline
			$ 0.5$ & $2749.89 $ & $-16.9039$\\\hline
			$0.7$  &$2667.64$ & $-16.2963$ \\\hline
			$ 0.9$ &  $2327.33$ & $-13.6055$\\\hline
			$ 0.9999$ &  $720.16$ & $-1.7242$\\\hline
		\end{tabular}
		\centering
		\begin{tabular}{|c|c|c|}\hline
			\multicolumn{3}{|c|}{$L= 100$} \\\hline
			$\kappa_0$ & $\mathcal{I}$ & $\det(\mathcal{D})$\\\hline
			$0.1$  &$22159.9 $ & $-68.2148$\\\hline
			$ 0.3$  & $22145.9 $ & $-68.1444$\\\hline
			$ 0.5$ & $21999.1 $ & $-67.6155$\\\hline
			$0.7$  &$21341.1$ & $-65.1854$\\\hline
			$ 0.9$ &  $18618.7$ & $-54.4218$\\\hline
			$ 0.9999$  &  $5761.48$ & $-6.8758$\\\hline
		\end{tabular}
		
	\end{table}

	\indent Using Maple program, we can plot the behavior of $\langle\mathcal{L}^{-1}1,1\rangle$ in terms of the modulus $k\in(0,1)$ for the case $L=20$. 
	
	\begin{figure}[!h]\begin{center}
			\includegraphics[scale=0.3]{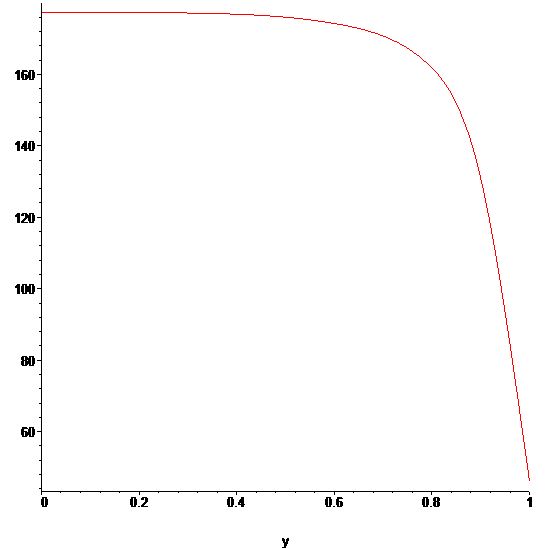}\label{Fig2}
			\caption{\small  Graphic of $\langle\mathcal{L}^{-1}1,1\rangle$ for $L=20$.}
		\end{center}
	\end{figure}

	\subsection{Zero mean periodic waves for the critical KdV - Proof of Theorem $\ref{teoest}$-b)} In what follows, we still consider $A=0$ and $g(s)=s^{5}$ in equation $(\ref{odekdv})$. Our intention is to give a complete scenario for the spectral stability in this case.\\
	\indent Let $L>0$ be fixed. An explicit even periodic wave satisfying the minimization problem in $(\ref{infB1})$ is given by
	\begin{equation}\label{cnoid4kdv}
		\phi(x)=\frac{a{\rm cn}\left(\frac{4K(k)}{L}x,k\right)}{\sqrt{1-b{\rm sn}^2\left(\frac{4K(k)}{L}x,k\right)}},
	\end{equation}
	where $a$, $b$ and $\omega$ also depend smoothly on the elliptic modulus $k\in(0,1)$ as
	\begin{equation}\label{ak}
		a=\frac{2[(2-k^2+2\sqrt{k^4-k^2+1})K(k)^2L^2]^{\frac{1}{4}}}{L},
	\end{equation}
	\begin{equation}\label{bk}
		b=-1+k^2-\sqrt{k^4-k^2+1},
	\end{equation}
	and
	
	\begin{equation}\label{wk}
		\omega=\frac{16K(k)^2\sqrt{k^4-k^2+1}}{L^2}.
	\end{equation}

\begin{obs}
	Let $L>0$ be fixed. Similarly as in the case of positive solutions, we also have the basic estimate $\omega>\frac{4\pi^2}{L^2}$ for the existence of periodic waves with cnoidal profile in $(\ref{cnoid4kdv})$. Moreover, when  $\omega>\frac{4\pi^2}{L^2}$, we obtain by the standard ODE theory that $\phi$ in $(\ref{cnoid4kdv})$ is the unique mean zero solution which solves equation $(\ref{odekdv})$ with $A=0$ and $g(s)=s^5$. The uniqueness of solutions gives us that $\phi$ in $(\ref{cnoid4kdv})$ satisfies the minimization problem $(\ref{infB1})$.
\end{obs}
	For a fixed $\omega_0>0$, we have already determined in the last subsection that the associated linearized operator around the periodic wave $\phi$ given by $(\ref{cnoid4kdv})$ satisfies $n(\mathcal{L}_{(\omega_0,0)})=2$ and $\ker(\mathcal{L}_{(\omega_0,0)})=[\phi']$ (see Theorem $\ref{teo1}$ and Proposition $\ref{propsimp}$). Therefore, we obtain the same spectral properties for all $(\omega,A)$ in an open subset contained in the parameter regime. From Theorem $\ref{teo2}$ we guarantee the existence of a smooth
	surface $$(\omega,A)\in\mathcal{O}\mapsto\phi_{(\omega,A)}\in
	H_{per,e}^n([0,L]),\ \ \ \ n\gg1,$$ of periodic waves, all of them
	with the same period $L>0$.\\
	\indent Since we have constructed the smooth surface
	$(\omega,A)\in\mathcal{O}\mapsto\phi_{(\omega,A)}\in
	H_{per,e}^2([0,L])$ of even periodic waves which solves the
	nonlinear differential equation $(\ref{travkdv})$ with $p=4$, the
	next step is to calculate $n(\mathcal{I})$ and $n(\mathcal{D})$. Table below shows us the behavior of the quantity $\mathcal{I}$ for some (fixed) values of $L>0$.
	
	\begin{table}[h]
		\centering
		\begin{tabular}{|c|c|}\hline
				\multicolumn{2}{|c|}{$L= 2\pi$} \\\hline
			$\kappa_0$& $\mathcal{I}$ \\\hline
			$ 0.0001$ &   $ -0.47290$\\\hline
			$0.1$ &  $ -0.47296$\\\hline
			$ 0.3$ &  $ -0.4643$\\\hline
			$ 0.5$ &$ -0.3886$\\\hline
			$0.7$ &  $-0.0985$ \\\hline
			$0.739$ &  $-0.0024$ \\\hline
			$0.746$ &  $0.0001$ \\\hline
			$ 0.9$ &  $ 0.4919 $\\\hline
			$ 0.9999$ &  $ 0.9958 $\\\hline
			\end{tabular}
		\centering
		\begin{tabular}{|c|c|}\hline
				\multicolumn{2}{|c|}{$L= 20$} \\\hline
			$\kappa_0$& $ \mathcal{I}$ \\\hline
			$ 0.0001$ &   $-16.2331$\\\hline
			$0.1$ & $-16.2298 $\\\hline
			$ 0.3$ & $-15.9064$\\\hline
			$ 0.5$ &  $ -13.3279 $\\\hline
			$ 0.7$ &  $-3.6809 $\\\hline
			$ 0.744$ &  $-0.070 $\\\hline
			$0.7449$ & $0.0095$\\\hline
			$ 0.9$ &  $15.7314$\\\hline
			$ 0.9999$ &   $44.3814$\\\hline	
		\end{tabular}
		\centering
		\begin{tabular}{|c|c|}\hline
			\multicolumn{2}{|c|}{$L= 50$} \\\hline
			$\kappa_0$& $\mathcal{I}$ \\\hline
			$ 0.0001$ &  $ -255.07$\\\hline
			$0.1$ &  $-255.01 $\\\hline
			$ 0.3$ &  $-249.89$\\\hline
			$ 0.5$ & $  -209.40 $\\\hline
			$0.7$ & $-58.24$\\\hline
			$0.74521$ & $-0.0263$\\\hline
			$0.74523$ & $0.0017$\\\hline
			$ 0.9$ &   $ 245.608$\\\hline
			$ 0.9999$ &   $ 71.1702$\\\hline
			
		\end{tabular}
		\centering
		\begin{tabular}{|c|c|}\hline
		\multicolumn{2}{|c|}{$L= 100$} \\\hline
		$\kappa_0$& $\mathcal{I}$ \\\hline
		$ 0.0001$ &  $-2042.21$\\\hline
		$0.1$ & $ -2041.74 $\\\hline
		$ 0.3$ &  $-2000.67 $\\\hline
		$ 0.5$ & $-1676.52 $\\\hline
		$0.7$ & $-466.79$\\\hline
		$0.74528$ & $-0.1066$\\\hline
		$0.74529$ & $0.0045$\\\hline
		$ 0.9$ &  $ 1964.63$\\\hline
		$ 0.9999$ &  $813.37$\\\hline
			
		\end{tabular}
		
	\end{table}

	Now, we plot the graphic of $\langle\mathcal{L}^{-1}1,1\rangle$ for the case $L=2\pi$.
	\newpage
	
	\begin{figure}[!h]\begin{center}
			\includegraphics[scale=0.3]{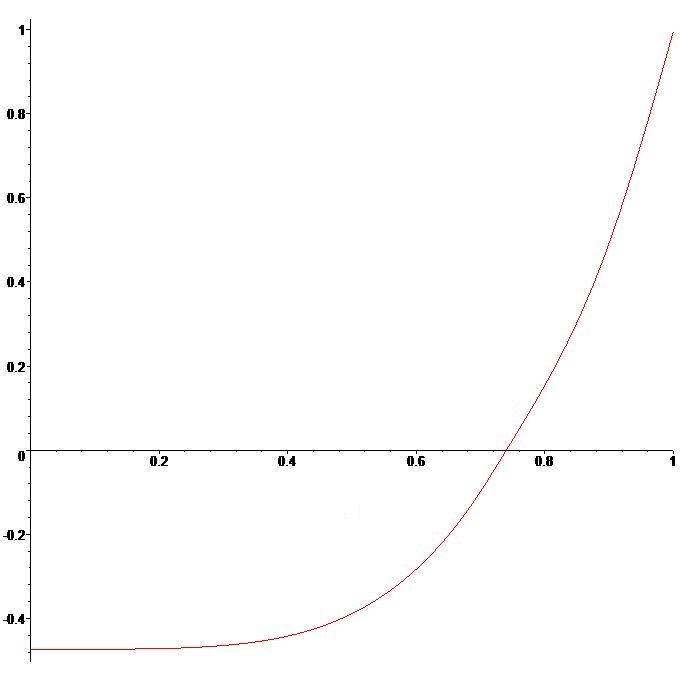}\label{Fig3}
			\caption{\small  Graphic of $\langle\mathcal{L}^{-1}1,1\rangle$ for $L=2\pi$.}
		\end{center}
	\end{figure}

	It remains to calculate $\det(\mathcal{D})$. For the case $(\omega,A)=(\omega,0),$ we see that $\langle\mathcal{L}^{-1}\phi,1\rangle=-\frac{d}{d\omega}\int_0^L\phi dx=0$. If $\mathcal{I}\neq0$, we can use $(\ref{Dmatrix})$ to obtain $$\det(\mathcal{D})=-\frac{1}{2}\ \frac{d}{d\omega}\int_0^L\phi^2dx.$$
	
	\indent Formula (411.03) in \cite{byrd} gives us that
	\begin{eqnarray}\label{norm}\int_0^L\phi^2 dx&=&\frac{a^2L}{K(k)}\int_0^{K(k)}\frac{\textrm{cn}^2(u,k)}{1-b\ \textrm{sn}^2(u,k)}du\nonumber\\
		\nonumber\\
		&=&\frac{a^2L}{K(k)}\ \frac{\pi(1-b)\ [1-\Lambda_0(\beta,k)]}{2\sqrt{b(1-b)(b-k^2)}}\\
		\nonumber\\
		&=&\frac{2\pi\sqrt{(k^2-2b)(1-b)}\ [1-\Lambda_0(\beta,k)]}{\sqrt{b(b-k^2)}}:=\tau(k).\nonumber\end{eqnarray}
	Here, $\Lambda_0$ indicates de Lambda Heumann function defined by
	$$\Lambda_0(\beta,k)=\frac{2}{\pi}[E(k)F(\beta,k')+K(k)E(\beta,k')-K(k)F(\beta,k')],$$  
	where $\beta=\arcsin\displaystyle\left(\frac{1}{\sqrt{1-b}}\right)$ and $ k'=\sqrt{1-k^2}.$
	
	\indent Next, for all $k\in(0,1)$ we have 
	\begin{equation}\label{dwdk}\frac{d\omega}{dk}=-\frac{16 K(k)\ [K(k)\ (k^4-3k^2+2)-2E(k)\ (k^4-k^2+1)]}{L^2k(1-k^2)\sqrt{k^4-k^2+1}}>0.\end{equation}
	Thus, 
	
	\begin{equation}\label{dwnorn}
		\frac{d}{d\omega}\int_0^L\phi^2 dx=\displaystyle\frac{\frac{d}{dk}\int_0^L\phi^2 dx}{\frac{d\omega}{dk}}=\frac{\tau'(k)}{\frac{d\omega}{dk}}.
	\end{equation}
	
	\indent We can plot the behavior of $\tau'$ in terms of the modulus $k\in(0,1)$ to conclude from $(\ref{dwdk})$ and $(\ref{dwnorn})$ that $\frac{d}{d\omega}\int_0^L\phi^2\ dx>0.$ 
	
	\begin{figure}[!h]\begin{center}
			\includegraphics[scale=0.35]{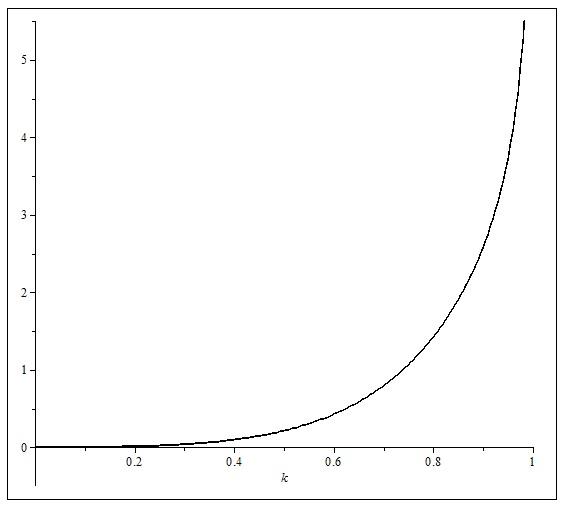}\label{Fig1}
			\caption{\small  Graphic of $\tau'(k)$.}
		\end{center}
	\end{figure}

	\indent Let $L>0$ be fixed. One sees that $\det(\mathcal{D})<0$ for all $\omega>\frac{4\pi^2}{L^2}$ and $A\approx 0$. Previous tables give us a threshold value $k_0\approx 0.745 $ satisfying 
	$\mathcal{I}=\langle\mathcal{L}^{-1}1,1\rangle=0$ at $k=k_0$ with $\mathcal{I}<0$ if $k\in(0,k_0)$ and $\mathcal{I}>0$ if $k\in(k_0,1)$. Therefore, we can apply Corollary $\ref{coroest}$ to conclude that $\phi$ is spectrally stable if $k\in(0,k_0)$ and spectrally unstable if $k\in(k_0,1)$.

	\subsection{The Gardner equation - Proof of Theorem $\ref{teoestG}$.} Now, we apply the arguments in the previous sections to study the spectral stability for the Gardner equation expressed in a general form as
	
	\begin{equation}\label{gardnereq1}
		u_t+\alpha_1 (u^2)_x+\alpha_2 (u^3)_x+u_{xxx}=0,
	\end{equation}
	where $\alpha_1$ and $\alpha_2$ are non-negative constants satisfying $\alpha_1^2+\alpha_2^2\neq0$. When $\alpha_2=0$, equation $(\ref{gardnereq1})$ reduces to the well KdV equation while $\alpha_1=0$, the same equation provides us the modified KdV equation. Results of spectral/orbital stability of periodic waves in both cases have been discussed in \cite{AN1}, \cite{AN2}, \cite{DK}, and references therein. \\
	\indent Let $L>0$ be fixed and consider $\alpha_1=\alpha_2=1$. Explicit periodic waves $\phi$ for this equation can be determined as
	\begin{equation}\label{cngardner}
		\phi(x)=-\frac{1}{3}+b{\rm cn}\left(\frac{4K(k)}{L}x,k\right),
	\end{equation}
	where $b$, $\omega$ and $A$ are given by
	\begin{equation}\label{ak1}
		b=\frac{4\sqrt{2}kK(k)}{L},
	\end{equation}
	
	\begin{equation}\label{wk1}
		\omega=-\frac{1}{3}-\frac{16K(k)^2(1-2k^2)}{L^2},
	\end{equation}
	and
	\begin{equation}\label{Ak1}
		A=\frac{1}{27}+\frac{144K(k)^2(1-2k^2)}{27L^2}.
	\end{equation}
	Now, let us define $\varphi(x)=\frac{1}{3}+\phi(x)=b{\rm cn}\left(\frac{4K(k)}{L}x,k\right)$. We see that $\varphi$ solves the equation
	\begin{equation}\label{mkdv9}
		-\varphi''+\left(\omega+\frac{1}{3}\right)\varphi-\varphi^3=0,
	\end{equation}
	that is, $\varphi$ is a periodic wave with cnoidal profile for the modified KdV equation.\\
	\indent Let $\mathcal{L}_{\varphi}$ be the linearized operator around $\varphi$ for the equation $(\ref{mkdv9})$ with $g(s)=s^3$. It is a surprising fact that
	$$\mathcal{L}_{\varphi}=-\partial_x^2+\omega+\frac{1}{3}-3\varphi^2=-\partial_x^2+\omega-2\phi-3\phi^2=\mathcal{L},$$
	where $\mathcal{L}$ is the linearized operator around $\phi$ for the Gardner equation $(\ref{gardnereq1})$. Using the arguments in \cite{AN1} and \cite{DK}, we see that $n(\mathcal{L})=n(\mathcal{L}_{\varphi})=2$ and $\ker(\mathcal{L})=\ker(\mathcal{L}_{\varphi})=[\phi']$.\\
	\indent Now, since $\mathcal{L}_{\varphi}=\mathcal{L}$, one sees that $\langle\mathcal{L}_{\varphi}^{-1}1,1\rangle=\langle\mathcal{L}^{-1}1,1\rangle$. In addition, we obtain similarly as determined in \cite{AN1} and \cite{DK} that $\langle\mathcal{L}^{-1}1,1\rangle<0$ for $k\in (0,k_0)$ and $\langle\mathcal{L}^{-1}1,1\rangle>0$ for $k\in(k_0,1)$, where $k_0\approx 0.909$.\\
	\indent It remains to calculate $\mathcal{D}$ in this specific case. First, we deal with
	\begin{equation}\label{norma1}\int_0^L\phi^2 dx=\frac{L}{9}-\frac{2b}{3}\int_0^L\textrm{cn}\displaystyle\left(\frac{4K(k)}{L}x,k\right) dx+b^2\int_0^L\textrm{cn}^2\displaystyle\left(\frac{4K(k)}{L}x,k\right) dx. \end{equation}
	
	\indent The middle term on the right-hand side of $(\ref{norma1})$ has zero mean since
	\begin{equation}\label{int01}\int_0^L\textrm{cn}\displaystyle\left(\frac{4K(k)}{L}x,k\right) dx=\frac{L}{2K(k)}\int_0^{2K(k)}\textrm{cn}(u,k)\ du=0.\end{equation} 
	In addition, the last term containing the quadratic power can be simplified as
	\begin{equation}\label{int02}\int_0^L\textrm{cn}^2\displaystyle\left(\frac{4K(k)}{L}x,k\right) dx=\frac{L}{K(k)}\int_0^{K(k)}\textrm{cn}^2(u,k)\ du=\frac{L[E(k)-(1-k^2)K(k)]}{k^2K(k)}.\end{equation}
	
	\indent Collecting the arguments contained in (\ref{ak1}), (\ref{norma1}), (\ref{int01}), and (\ref{int02}), we conclude that
	\begin{equation}\label{norma2}\int_0^L\phi^2 dx=\frac{L}{9}+\frac{32 K(k)[E(k)-(1-k^2)K(k)]}{L}. \end{equation}
	
	\indent In order to calculate $\frac{d}{d\omega}\int_0^L\phi^2 dx$, we need to observe first that for all $k\in(0,1)$ we have
	\indent  \begin{equation}\label{dwdk2}\frac{d\omega}{dk}=-\frac{32K(k)[(1-2k^2)E(k)-(1-k^2)K(k)]}{k(1-k^2)L^2}>0.\end{equation}
	Thus, we are in position to give a convenient expression for $\frac{d}{d\omega}\int_0^L\phi^2 dx$ using the chain rule as 
	\begin{equation}\label{derphi2}\frac{d}{d\omega}\int_0^L\phi^2 dx=\frac{\frac{d}{dk}\int_0^L\phi^2dx}{\frac{d\omega}{dk}}.\end{equation}
	By (\ref{norma2}), we obtain for all $k\in (0,1)$ that \begin{equation}\label{dnorm2}\frac{d}{dk}\int_0^L\phi^2 dx=-\frac{32[(1-k^2)K(k)(2E(k)-K(k))-E(k)^2]}{k(1-k^2)L}>0.\end{equation}

	\indent Finally, by (\ref{dnorm2}) and (\ref{dwdk2}) we deduce \begin{equation}\label{positivephi2}\frac{1}{2}\frac{d}{d\omega}\int_0^L\phi^2 dx=\frac{\frac{d}{dk}\int_0^L\phi^2 dx}{2\frac{d\omega}{dk}}>0.\end{equation}
	
	\indent We are in position to calculate $\mathcal{D}$. We have already determined  that $\langle\mathcal{L}_{\varphi}^{-1}1,1\rangle=\langle\mathcal{L}^{-1}1,1\rangle$. Similarly, since $\varphi$ has zero mean it follows that $\langle\mathcal{L}_{\varphi}^{-1}\varphi,1\rangle=\langle\mathcal{L}^{-1}\phi,1\rangle=-\frac{d}{d\omega}\int_0^L\phi dx=0$. In addition, a straightforward calculation also gives us that $\langle\mathcal{L}_{\varphi}^{-1}\varphi,\varphi\rangle=\langle\mathcal{L}^{-1}\phi,\phi\rangle=-\frac{1}{2}\frac{d}{d\omega}\int_0^L\phi^2 dx$. Thus, for $k\neq k_0\approx 0.909$ one has from $(\ref{positivephi2})$ that $\det(\mathcal{D})$ is 
	$$\det(\mathcal{D})=-\frac{1}{2}\frac{d}{d\omega}\int_0^L\phi^2 dx<0.$$ Corollary $\ref{coroest}$ can be applied to deduce the spectral stability of $\phi$ when $k\in(0,k_0)$ and the spectral instability when $k\in (k_0,1)$.

	\section*{Acknowledgments} S.A. was supported by CAPES.  F.N. is partially supported by Funda\c{c}\~ao Arauc\'aria 002/2017, CNPq 304240/2018-4 and CAPES MathAmSud 88881.520205/2020-01.

\end{document}